\documentclass{article}
\usepackage{amsmath} 
\usepackage{amsthm} 
\usepackage{amssymb} 
 \makeatletter
 \def\BlackBoxes{\global\overfullrule5\p@}  %Original 
 \makeatother
 \makeatother
 \makeatletter
 \@addtoreset{equation}{section}
 \makeatother
\theoremstyle{plain} 
\newtheorem{theorem}{Theorem} 
\newtheorem{corollary}{Corollary} 
\newtheorem{lemma}{Lemma} 
\theoremstyle{definition} 
\newtheorem{example}{Example}  
\newtheorem{examples}{Examples}  
\newtheorem{remark}{Remark} 
\newtheorem{remarks}{Remarks}

\begin{document} 
\date{} 
\title{Population Dynamics in Hostile Neighborhoods}
\author{Herbert Amann}
\maketitle
\begin{center} 
{\small 
Dedicated to Juli\'{a}n L\'{o}pez-G\'{o}mez,\\  
a~mathematical friend for many years.} 
\end{center}  
\begin{abstract} 
A~new class of quasilinear reaction-diffusion equations is 
introduced for which the mass flow never reaches the boundary. 
It is proved that the initial value problem is well-posed in 
an appropriate weighted Sobolev space setting. 
\end{abstract} 
\makeatletter
\def\blfootnote{\xdef\@thefnmark{}\@footnotetext}
\makeatother 
 \newcommand*{\al}{\alpha}
 \newcommand*{\ba}{\beta}
 \newcommand*{\Da}{\Delta}
 \newcommand*{\da}{\delta}
 \newcommand*{\Ga}{\Gamma}
 \newcommand*{\ga}{\gamma}
 \newcommand*{\ia}{\iota}
 \newcommand*{\lda}{\lambda}
 \newcommand*{\Om}{\Omega}
 \newcommand*{\oO}{\ol{\Omega}} 
 \newcommand*{\pO}{{\pl\Omega}}
 \newcommand*{\sa}{\sigma}
 \newcommand*{\ta}{\theta}
 \newcommand*{\ve}{\varepsilon}
 \newcommand*{\vp}{\varphi}
 \newcommand*{\EeEn}{(E_1,E_0)}
 \newcommand*{\Mg}{(M,g)}
 \newcommand*{\BN}{{\mathbb N}}
 \newcommand*{\BR}{{\mathbb R}}
 \newcommand*{\cA}{{\mathcal A}}
 \newcommand*{\cB}{{\mathcal B}}
 \newcommand*{\cL}{{\mathcal L}}
 \newcommand*{\gF}{{\mathfrak F}}
 \newcommand*{\bal}{\begin{aligned}}
 \newcommand*{\eal}{\end{aligned}}
 \newcommand*{\uti}[1]{(#1)\space\space}
 \newcommand*{\qa}{,\qquad}
 \newcommand*{\qb}{,\quad}
 \newcommand*{\mf}[1]{\boldsymbol{#1}}        
 \newcommand*{\hb}[1]{\hbox{$#1$}} %do not remove 
 %where next 3 commands are used
 \newcommand*{\sdot}{\!\cdot\!}
 \newcommand*{\sn}{\kern1pt|\kern1pt}
 \newcommand*{\ssm}{\!\setminus\!}
 \newcommand*{\vsdot}{\hbox{$\vert\sdot\vert$}}
 \newcommand*{\Vsdot}{\hbox{$\Vert\sdot\Vert$}}
 \newcommand*{\npbd}{\postdisplaypenalty=10000} %USE ONLY INSIDE DISPLAYS! 
 \newcommand*{\pr}{\hbox{$(\cdot,\cdot)$}}
 \newcommand*{\prsn}{\hbox{$(\cdot\sn\cdot)$}}
 \newcommand*{\pw}{\hbox{$\dl{}\sdot{},{}\sdot{}\dr$}}
 \newcommand*{\ol}{\overline}
 \newcommand*{\ul}{\underline}
 \newcommand*{\ph}{\phantom}
 \newcommand*{\dl}{\langle}
 \newcommand*{\dr}{\rangle}
 \newcommand*{\ra}{\rightarrow}
 \newcommand*{\hr}{\hookrightarrow}
 \newcommand*{\dist}{\mathop{\rm dist}\nolimits}
 \newcommand*{\tdiv}{\mathop{\rm div}\nolimits}
 \newcommand*{\grad}{\mathop{\rm grad}\nolimits}
 \newcommand*{\loc}{{\rm loc}}
 \newcommand*{\idmi}[1]{{\mbox{\scriptsize$#1${\rm-}}}}     
 \newcommand*{\is}{\subset}
 \newcommand*{\isis}{\Subset} 
 \newcommand*{\bt}{\bullet}
 \newcommand*{\es}{\emptyset}
 \newcommand*{\iy}{\infty}
 \newcommand*{\mt}{\mapsto}
 \newcommand*{\pl}{\partial}
 \newcommand*{\coW}{\kern-1pt}
 %The following definition of ALIGNED is here
 %to make sure that always the same one is used.
 %I encountered a few ugly surprises, since different versions of TEX
 %use different definitions.
 %Since this paper contains some VERY delicate constructions,
 %this definition is ABSOLUTELY necessary
 %It will be valid only in THIS paper and not affect anything else
 %-------------------------------------------------------------------------
 %-------------------------------------------------------------------------
 %-------------------------------------------------------------------------
 \makeatletter
 \newif\ifinany@
 \newcount\column@
 \def\column@plus{%
    \global\advance\column@\@ne
 }
 \newcount\maxfields@
 \def\add@amps#1{%
    \begingroup
        \count@#1
        \DN@{}%
        \loop
            \ifnum\count@>\column@
                \edef\next@{&\next@}%
                \advance\count@\m@ne
        \repeat
    \@xp\endgroup
    \next@
 }
 \def\Let@{\let\\\math@cr}
 \def\restore@math@cr{\def\math@cr@@@{\cr}}
 \restore@math@cr
 \def\default@tag{\let\tag\dft@tag}
 \default@tag

 \newbox\strutbox@
 \def\strut@{\copy\strutbox@}
 \addto@hook\every@math@size{%
  \global\setbox\strutbox@\hbox{\lower.5\normallineskiplimit
         \vbox{\kern-\normallineskiplimit\copy\strutbox}}}

 \renewcommand{\start@aligned}[2]{%        
    \RIfM@\else
        \nonmatherr@{\begin{\@currenvir}}%
    \fi
    \null\,%
    \if #1t\vtop \else \if#1b \vbox \else \vcenter \fi \fi \bgroup
        \maxfields@#2\relax
        \ifnum\maxfields@>\m@ne
            \multiply\maxfields@\tw@
            \let\math@cr@@@\math@cr@@@alignedat
        \else
            \restore@math@cr
        \fi
        \Let@
        \default@tag
        \ifinany@\else\openup\jot\fi
        \column@\z@
        \ialign\bgroup
           &\column@plus
            \hfil
            \strut@
            $\m@th\displaystyle{##}$%
           &\column@plus
            $\m@th\displaystyle{{}##}$%
            \hfil
            \crcr
 }
 \renewenvironment{aligned}[1][c]{%        
    \start@aligned{#1}\m@ne
 }{%
    \crcr\egroup\egroup
 }
 \makeatother
 %-------------------------------------------------------------------------
 %-------------------------------------------------------------------------
 %-------------------------------------------------------------------------
\section{Introduction}\label{sec-I} 
It has been known since long that spacial interactions in population dynamics 
can be adequately modeled by systems of reaction-diffusion equations 
(see E.E.~Holmes et al.~\cite{HLBV94a}, A.~Okubo and 
S.A.~Levin~\cite{OkL01a}, or J.D.~Murray~\cite{Mur89a}, \cite{Mur03a}, for 
instance). In general, these systems possess a quasilinear structure and 
show an extremely rich qualitative behavior depending on the various 
structural assumptions which can meaningfully be imposed. 

\smallskip 
Reaction-diffusion equations are of great importance also in many other 
scientific areas as, for example, physics, chemistry, mechanical and 
chemical engineering, and the 
social sciences. Thus our mathematical results are not restricted to 
population dynamics. It is just a matter of convenience to describe the 
phenomenological background and motivation in terms of populations. 

\smallskip 
Throughout this paper, $\Om$~is a bounded domain in~$\BR^m$ with a smooth 
boundary~$\Ga$ lying locally on one side of~$\Om$. (In population dynamics, 
\hb{m=1},~$2$, or~$3$. But this is not relevant for what follows.) By~$\nu$ 
we denote the inner (unit) normal vector field on~$\Ga$ and use $\cdot$ or 
\hb{\prsn} for the Euclidean inner product in~$\BR^m$. 

\smallskip 
We assume that $\Om$~is occupied by $n$~different species described by their 
densities 
\hb{u_1,\ldots,u_n} and set 
\hb{u:=(u_1,\ldots,u_n)}. The spacial and temporal change of~$u$, that is, 
the (averaged) movement of the individual populations, is mathematically 
encoded in the form of conservation laws 
\begin{equation}\label{I.cl} 
\pl_tu_i+\tdiv j_i(u)=f_i(u) 
\quad\text{in}\quad 
\Om\times\BR_+ 
\qa 1\leq i\leq n. 
\end{equation}  
Here $j_i(u)$~is the (mass) flux vector, $f_i(u)$~the production rate of the 
\hbox{$i$-th} species, and \eqref{I.cl} is a mass balance law (e.g., 
S.R.~de~Groot and P.~Mazur~\cite{GrM84a}). 

\smallskip 
In order to get a significant model we have to impose constitutive 
assumptions on the \hbox{$n$-tuple} 
\hb{j(u)=\bigl(j_1(u),\ldots,j_n(u)\bigr)} of flux vectors. In population 
dynamics it is customary to build on phenomenological laws which are 
basically variants and extensions of Fick's law, and we adhere in this paper 
to that practice. Thus we assume that 
$$ 
j_i(u)=-a_i(u)\grad u_i, 
$$ 
where the `diffusion coefficient' 
\hb{a_i(u)\in C^1(\Om)} may depend on the interaction of some, or all, 
species, hence on~$u$. The fundamental assumption is then that 
\begin{equation}\label{I.ai} 
a_i(u)(x)>0 
\qa x\in\Om. 
\end{equation}  

\smallskip 
Besides of modeling the behavior of the populations in~$\Om$, their conduct 
on the boundary~$\Ga$ has also to be analyzed. We restrict ourselves to 
homogeneous boundary conditions. Then there are essentially two cases 
which are meaningful, namely the Dirichlet boundary condition 
$$ 
u_i=0 
\quad\text{on}\quad 
\Ga\times\BR_+ 
$$ 
or the no-flux condition 
$$ 
\nu\cdot j_i(u)=0 
\quad\text{on}\quad 
\Ga\times\BR_+ 
\npbd 
$$ 
for population~$i$, or combinations thereof. 

\smallskip 
In the standard mathematical theory of reaction-diffusion equations it is 
assumed that $a_i(u)$~is uniformly positive on~$\Om$. 
The focus of this paper is on the nonuniform case where 
$a_i(u)$~may tend to zero as we approach~$\Ga$. 

\smallskip 
In population dynamics nonuniformly positive one-population models have been 
introduced by W.S.C. Gurney and R.M.~Nisbet~\cite{GuN75a} and 
M.E.~Gurtin and R.C.~MacCamy~\cite{GuM77a}. Arguing that the population 
desires to avoid overcrowding, they arrive at flux vectors of the form 
\hb{j(u)=-u^k\grad u} with 
\hb{k\geq1}. Thus their models are special instances of the porous media 
equation. Here the diffusion coefficient degenerates, in particular, 
near the Dirichlet boundary. 

\smallskip 
Ever since the appearance of the pioneering papers 
\cite{GuN75a},~\cite{GuM77a}, there have been numerous studies of the (weak) 
solvability of reaction-diffusion equations and systems exhibiting porous 
media type degenerations. We do not go into detail, since we propose a 
different approach. 

\smallskip 
In a series of papers, J.~L\'{o}pez-G\'{o}mez has studied (partly with 
coauthors) qualitative 
properties of \hbox{one-} and two-population models which he termed 
`degenerate' (see \cite{LoG16a}, \cite{LoMa17a}, \cite{LoMR17a}, 
\cite{DLo18a}, \cite{AALo20a}, and the references therein). In those works 
the term `degenerate', however, refers to the vanishing  on some open subset 
of~$\Om$ of the `logistic coefficient', which is part of~$f(u)$. 

\smallskip 
For the following heuristic discussion, which describes the essence of our 
approach, we can assume that 
\hb{n=1} and that $a$~is independent of~$u$. 

\smallskip 
If the Dirichlet boundary condition holds, then the population gets extinct 
if it reaches~$\Ga$. In the case of the no-flux condition, $\Ga$~is 
impenetrable, that is, the species can neither escape through the boundary 
nor can it get replenishment from the outside. In this sense we can say that 
the `population lives in a hostile neighborhood'. 

\smallskip 
No slightly sensible species will move toward places where it is 
endangered to get killed, nor will it run head-on against an impenetrable 
wall. Instead, it will slow down drastically if it comes near such places. 
In mathematical terms this means that the flux in the normal direction 
has to decrease to zero near~$\Ga$. To achieve this, the diffusion 
coefficient~$a$ has to vanish sufficiently rapidly at~$\Ga$. 

\smallskip 
To describe more precisely what we have in mind, we study the motion 
of the population in a normal collar neighborhood of~$\Ga$. This means that 
we fix 
\hb{0<\ve\leq1} such that, setting 
$$ 
S:=\bigl\{\,q+y\nu(q)\ ;\ 0<y\leq\ve,\ q\in\Ga\,\bigr\}, 
$$ 
the map 
\begin{equation}\label{I.ph} 
\vp\colon\ol{S}\ra[0,\ve]\times\Ga 
\qb q+y\nu(q)\mt(y,q) 
\end{equation}  
is a smooth diffeomorphism. Note that 
$$ 
y=\dist(x,\Ga) 
\qb x=q+y\nu(q)\in S. 
$$ 
We extend~$\nu$ to a smooth vector field on~$S$, again denoted by~$\nu$, 
by setting 
\begin{equation}\label{I.nu} 
\nu(x):=\nu(q) 
\qa x=q+y\nu(q)\in S. 
\end{equation}  
Then the normal derivative 
\begin{equation}\label{I.dnu} 
\pl_\nu u(x):=\nu(x)\cdot\grad u(x) 
\npbd 
\end{equation}  
is well-defined for 
\hb{x\in S}. 

\smallskip 
As usual, we denote by~$\vp_*$ the push-forward and by~$\vp^*$ the pull-back 
by~$\vp$ of functions and tensors, in particular of vector fields. Then 
\begin{equation}\label{I.phg} 
\vp_*(a\grad)=(\vp_*a)\frac\pl{\pl y}\oplus(\vp_*a)\grad_\Ga 
\quad\text{on }\quad 
N:=(0,\ve]\times\Ga, 
\end{equation}  
where $\grad_\Ga$~is the surface gradient on~$\Ga$ with respect to the 
metric induced by the Euclidean metric on~$\oO$. Note that, 
by \eqref{I.nu} and \eqref{I.dnu}, 
\begin{equation}\label{I.dd} 
\pl_y v=\pl_\nu u 
\qb v=\vp_*u. 
\end{equation}  
Set 
\hb{\Ga_y:=\vp^{-1}\bigl(\{y\}\times\Ga\bigr)}. Then 
\hb{\grad_{\Ga_y}=\grad_\Ga}. Hence we obtain from \eqref{I.phg} and 
\eqref{I.dd} that 
\begin{equation}\label{I.ag} 
a\grad u=(a\pl_\nu u)\nu\oplus a\grad_\Ga u 
\quad\text{on }S. 
\end{equation}  
Thus, if we want to achieve that the flux 
\hb{j(u)=-a\grad u} decays in the normal direction, but not necessarily 
in directions 
parallel to the boundary, we have to replace \eqref{I.ag} by 
$$ 
(a_1\pl_\nu u)\nu\oplus a\grad_\Ga u, 
$$ 
where $a_1$~tends to zero as $x$~approaches~$\Ga$. This we effectuate by 
replacing~$j(u)$ by 
\begin{equation}\label{I.js} 
j^s(u):=-\bigl((a\rho^{2s}\pl_\nu u)\nu\oplus a\grad_\Ga u\bigr) 
\qa u\in C^1(S), 
\end{equation}  
for some 
\hb{s\geq1}, where 
\hb{0<\rho\leq1} on~$S$ and 
\hb{\rho(x)=\dist(x,\Ga)} for $x$ near~$\Ga$. Then, irrespective of the size 
of~$\grad u$, \;$j^s(u)$~decays to zero as we approach~$\Ga$. The `speed' of 
this decay increases if $s$~gets bigger. Note, however, that the component 
orthogonal to~$\nu$ is the same as in \eqref{I.ag}. This reflects the fact, 
known to everyone who has been hiking in high mountains---in the Swiss Alps, 
for example(!)---that one can move forward along a level line path in 
front of a steep slope with essentially the same speed as this can be 
done in the flat country. On the other hand, one slows down drastically---and 
eventually gives up---if one tries to go to the top along a line of 
steepest ascent. 

\smallskip 
In the next section we give a precise definition of the class of degenerate 
equations which we consider. Section~\ref{sec-T} contains the definition of 
the appropriate weighted Sobolev spaces. In addition, we present the basic 
maximal regularity theorem for linear degenerate parabolic initial value 
problems. 

\smallskip 
The main result of this paper is Theorem~\ref{thm-Q.RD} which is proved in 
Section~\ref{sec-Q}. It guarantees the local well-posedness of quasilinear 
degenerate reaction-diffusion systems. In the last section we present some 
easy examples, discuss the differences between the present and the 
classical approach, and suggest possible directions of further research. 
\section{Degenerate Reaction-Diffusion Operators}\label{sec-S} 
Let $\Mg$ be a Riemannian manifold. Then $\grad_g$, resp.~$\tdiv_g$, 
denotes the gradient, resp.\ divergence, operator on~$\Mg$. The Riemannian  
metric on~$\Ga$, induced by the Euclidean metric on~$\oO$, is written~$h$. 
Then 
\hb{\ol{N}=[0,\ve]\times\Ga} is endowed with the metric 
\hb{g_N:=dy^2+h}. 

\smallskip 
We fix 
\hb{\chi\in C^\iy\bigl([0,\ve],[0,1]\bigr)} satisfying 
$$ 
\chi(y)= 
\left\{
\bal 
{}%  do not remove!!!!
&1,     &\quad 0         &\leq y\leq\ve/3,\cr 
&0,     &\quad 2\ve/3    &\leq y\leq\ve.
\eal 
\right.
$$ 
Then 
$$ 
r(y):=\chi(y)y+1-\chi(y) 
\qa 0\leq y\leq\ve.  
$$ 
We set 
$$ 
S(j):=\vp^{-1}\bigl((0,j\ve/3]\times\Ga\bigr) 
\qa j=1,2, 
$$ 
and 
\hb{\rho:=\vp^*r=r\circ\vp^{-1}}. Then 
\hb{\rho\in C^\iy\bigl(S,(0,1]\bigr)} and 
\begin{equation}\label{S.rh} 
\rho(x)=  
\left\{ 
\bal 
{}%  do not remove!!!!
&\dist(x,\Ga),  &&\quad x\in S(1),\cr 
&1,             &&\quad x\in S\ssm S(2).
\eal 
\right. 
\end{equation}  

\smallskip 
Given a linear differential operator~$\cB$ on~$S$, we denote by~$\vp_*\cB$ 
its `representation in the variables 
\hb{(y,q)\in N}'. Thus $\vp_*\cB$, the push-forward of~$\cB$, is the 
linear operator on~$N$ defined by 
$$ 
(\vp_*\cB)w:=\vp_*\bigl(\cB(\vp^*w)\bigr) 
\qa w\in C^\iy(N). 
$$ 

\smallskip 
First we consider a single linear operator, that is, 
\hb{n=1} and 
$$ 
\cA v:=-\tdiv(a\grad v) 
$$ 
with 
\begin{equation}\label{S.a} 
a\in C^1(\Om) 
\qb a(x)>0\text{ for }x\in\oO.
\end{equation}  

\smallskip 
We set 
\hb{\ol{a}:=\vp_*a\in C^1(N)} and 
\hb{\ol{\cA}:=\vp_*\cA}. Then we find 
$$ 
\ol{\cA}w=-\tdiv_{g_N}(\ol{a}\grad_{g_N}w) 
=-\pl_y(\ol{a}\pl_yw)-\tdiv_h(\ol{a}\grad_hw) 
$$ 
for 
\hb{w\in C^2(N)}. By pulling~$\ol{\cA}$ back to~$S$ we obtain the 
representation 
\begin{equation}\label{S.Au} 
\cA u=-\pl_\nu(a\pl_\nu u)-\tdiv_h(a\grad_hu) 
\qa u\in C^2(S), 
\npbd 
\end{equation}  
of 
\hb{\cA\sn S}. 

\smallskip 
We put 
\begin{equation}\label{S.U} 
U:=\Om\ssm S(2) 
\end{equation}  
and fix 
\hb{s\in[1,\iy)}. Then we define a linear operator~$\cA_s$ on~$\Om$ 
by setting 
\begin{equation}\label{S.Asv}  
\cA_sv:=-\tdiv_s(a\grad_sv) 
\qa v\in C^2(\Om),  
\end{equation}   
where 
$$ 
\tdiv_s(a\grad_sv):=  
\left\{ 
\bal 
{}%  do not remove!!!!
&\tdiv(a\grad v),                              &&\quad v\in C^2(U),\cr 
&\rho^s\pl_\nu(a\rho^s\pl_\nu v) 
 +\tdiv_h(a\grad_hv),                           &&\quad v\in C^2(S).
\eal 
\right. 
$$ 
It follows from \eqref{S.rh}, \eqref{S.Au}, and \eqref{S.U} that $\cA_sv$~is 
well-defined for 
\hb{v\in C^2(\Om)}. The map~$\cA_s$ is said to be a linear 
\emph{\hbox{$s$-degenerate} reaction-diffusion} (or \emph{divergence form}) 
\emph{operator on}~$\Om$. 
\begin{remark}\label{rem-S.g} 
It has been shown in~\cite{Ama20a} that the right approach to study 
differential operators which are \hbox{$s$-degenerate}, 
is to endow~$S$ with the metric 
$$ 
g_s:=\vp^*(r^{-2s}dy^2\oplus h). 
$$ 
Then  
$$ 
a\grad_{g_s}u=(a\rho^{2s}\pl_\nu u)\nu\oplus a\grad_hu 
\qa u\in C^1(S), 
$$ 
which equals~$-j^s(u)$ of~\eqref{I.js}. Furthermore, 
$$ 
\cA_su=-\tdiv_{g_s}(a\grad_{g_s}u) 
\qa u\in C^2(S). 
$$ 
Thus $\cA_s$~is a~`standard' linear 
reaction-diffusion operator if $S$~is endowed with the 
metric~$g_s$.\hfill$\qed$ 
\end{remark} 
\section{The Isomorphism Theorem}\label{sec-T} 
The natural framework for an efficient theory of strongly degenerate 
reaction-diffusion systems are weighted function spaces which we introduce 
now. We assume throughout that 
$$ 
\bt\quad 
1<p<\iy. 
$$ 
Suppose 
\hb{s\geq1} and 
\hb{k\in\BN}. For 
\hb{u\in C^k(S)} set 
$$  
v(y,q):=\vp_*u(y,q)=u\bigl(q+y\nu(q)\bigr) 
$$ 
and 
$$ 
\|u\|_{W_{\coW p}^k(S;s)} 
:=\sum_{i=0}^k\Bigl(\int_0^\ve\big\|\big(r(y)^s\pl_y\big)^iv(y,\cdot) 
\big\|_{W_{\coW p}^{k-i}(\Ga)}^pr(y)^{-s}\,dy\Bigr)^{1/p}. 
$$ 
Then the weighted Sobolev space~$W_{\coW p}^k(S;s)$ is the completion 
in~$L_{1,\loc}(S)$ of the subspace of smooth compactly supported functions 
with respect to the norm~% 
\hb{\Vsdot_{W_{\coW p}^k(S;s)}}. The \emph{weighted Sobolev 
space}~$W_{\coW p}^k(\Om;s)$ consists of all~$u$ belonging 
to~$L_{1,\loc}(\Om)$ with 
$$ 
u\sn S\in W_{\coW p}^k(\Om;s) 
\qb u\sn U\in W_{\coW p}^k(\Om). 
$$ 
It is a Banach space with the norm 
$$ 
u\mt\|u\sn S\|_{W_{\coW p}^k(S;s)} 
+\|u\sn U\|_{W_{\coW p}^k(U)}, 
\npbd 
$$ 
and 
\hb{L_p(\Om;s):=W_{\coW p}^0(\Om;s)}. Of course, $W_{\coW p}^k(U)$~is the 
usual Sobolev space. 

\smallskip 
To define weighted spaces of bounded \hbox{$C^k$~functions} we set 
\begin{equation}\label{T.Ck} 
\|u\|_{BC^k(S;s)} 
:=\sum_{i=0}^k\sup_{0<y<\ve} 
\big\|\big(r(y)^s\pl_y\big)^iv(y,\cdot)\big\|_{C^k(\Ga)}. 
\end{equation}  
The weighted  space~$BC^k(S;s)$ is the linear subspace of all 
\hb{u\in C^k(S)} for which the norm~\eqref{T.Ck} is finite. Then 
$BC^k(\Om;s)$ is the linear space of all 
\hb{u\in C^k(\Om)} with 
$$ 
u\sn S\in BC^k(S;s) 
\qb u\sn U\in BC^k(U). 
$$  
It is a Banach space with the norm 
\begin{equation}\label{T.SU} 
u\mt\|u\sn S\|_{BC^k(S;s)} 
+\|u\sn U\|_{BC^k(U)}. 
\end{equation}  

\smallskip 
The topologies of the weighted spaces $W_{\coW p}^k(\Om;s)$ and 
$BC^k(\Om;s)$ are independent of the particular choice of~$S$ (that is, of 
\hb{\ve>0}) and the cut-off function~$\chi$. 

\smallskip 
Let 
\hb{0<T<\iy} and set 
\hb{J:=[0,T]}. On 
\hb{\Om\times J} we introduce anisotropic weighted Sobolev spaces by 
$$ 
W_{\coW p}^{(2,1)}(\Om\times J;s) 
:=L_p\bigl(J,W_{\coW p}^2(\Om;s)\bigr) 
\cap W_{\coW p}^1\bigl(J,L_p(\Om;s)\bigr). 
$$ 

\smallskip 
We denote by~% 
\hb{\pr_{\ta,p}} the real interpolation functor of exponent  
\hb{\ta\in(0,1)}. Then we institute a Besov space by 
$$ 
B_p^{2-2/p}(\Om;s) 
:=\bigl(L_p(\Om;s),W_{\coW p}^2(\Om;s)\bigr)_{1-1/p,p}. 
$$ 
\begin{lemma}\label{lem-T.em} 
The weighted spaces possess the same embedding and interpolation properties 
as their non-weighted versions. In particular, 
\begin{equation}\label{T.BC} 
BC^1(\Om)\hr BC^1(\Om;s)\hr BC(\Om). 
\end{equation}  
\end{lemma}
\begin{proof} 
The first assertion is a consequence of Theorems 3.1 and~6.1 
of~\cite{Ama20a}. The first embedding of~\eqref{T.BC} is obvious from 
\hb{|\rho^s\pl_\nu u|\leq|\pl_\nu u|} and \eqref{T.SU}. 
It remains to observe that 
\hb{BC(\Om;s)=BC(\Om)}. 
\end{proof} 
The following theorem settles the well-posedness problem for the linear 
initial value problem 
$$ 
\bal 
\pl_tu+\cA_s u  &=f    &&\quad\text{on }    &\Om    &\times J, \cr
      \ga_0u    &=u_0  &&\quad\text{on }    &\Om    &\times\{0\}, 
\eal 
\npbd 
$$ 
where $\ga_0$~is the trace operator at 
\hb{t=0} and $\cA_s$~is given by \eqref{S.Asv}. 
\addtocounter{theorem}{1} 
\begin{theorem}\label{thm-T.MR} 
Let 
\hb{1\leq s<\iy} and 
\hb{0<T<\iy}. Assume that there exists 
\hb{\ul{\al}>0} such that 
\begin{equation}\label{T.ai} 
a\in BC^1(\Om;s) 
\quad\text{and}\quad 
a\geq\ul{\al}. 
\end{equation}  
Then the map 
\hb{(\pl_t+\cA_s,\,\ga_0)} is a topological isomorphism from 
$$ 
W_{\coW p}^{(2,1)}(\Om\times J;s) 
\quad\text{onto}\quad 
L_p(\Om\times J;s)\times B_p^{2-2/p}(\Om;s).  
$$ 
\end{theorem} 
\begin{proof} 
Note that 
\begin{equation}\label{T.U} 
\tdiv(a\grad v)=a\Da v+\dl da,\grad v\dr 
\quad\text{on }U, 
\end{equation}  
where 
\hb{\pw}~stands for duality pairings. On~$S$ we get 
\begin{equation}\label{T.r} 
\rho^s\pl_\nu(a\rho^s\pl_\nu v) 
=a(\rho^s\pl_\nu)^2v+(\rho^s\pl_\nu a)\rho^s\pl_\nu v 
\end{equation} 
and 
\begin{equation}\label{T.S} 
\tdiv_h(a\grad_hv)=a\Da_hv+\dl da,\grad_hv\dr, 
\npbd 
\end{equation} 
where $\Da_h$~is the Laplace--Beltrami operator on~$\Ga$. 

\smallskip 
Set 
\hb{R(t):=t^s} for 
\hb{0\leq t\leq1}. Then we deduce from \hbox{\eqref{T.ai}--\eqref{T.S}} 
and from Theorems 6.1 and~7.2 of~\cite{Ama20a} that $\cA_s$~is a 
\hbox{bc-regular} \hbox{$R$-degenerate} uniformly strongly elliptic 
differential operator on~$\Om$ in the sense of \cite[(1.6)]{Ama20a}. 
(Observe that the regularity condition~(1.9) in that paper is only 
sufficient and stronger than~\eqref{T.ai}.) Hence the assertion follows from 
Theorem~1.3 of~\cite{Ama20a}. 
\end{proof} 
\addtocounter{corollary}{2} 
\begin{corollary}\label{cor-T.MR} 
$\cA_s$~has maximal \hbox{$L_p(\Om;s)$~regularity}. 
\end{corollary} 
\begin{proof} 
\cite{Ama95a} or \cite{Ama20a}. 
\end{proof} 
\section{Quasilinear Degenerate Systems}\label{sec-Q} 
Now we turn to systems and consider quasilinear differential operators. 
Thus we assume: 
\begin{equation}\label{Q.ass} 
\bal 
{\rm(i)}\quad 
&1\leq s<\iy.\cr 
{\rm(ii)}\quad 
&X\text{ is a nonempty open subset of }\BR^n.\cr 
{\rm(iii)}\quad 
&a_i\in C^2\bigl(\oO\times X,\,(0,\iy)\bigr),\ 1\leq i\leq n. 
\eal 
\end{equation}  
Given 
\hb{u\in C^1(\Om,X)}, 
\begin{equation}\label{Q.ai} 
a_i(u)(x):=a_i\bigl(x,u(x)\bigr) 
\qa x\in\Om. 
\end{equation}  
Then 
$$ 
\cA_{i,s}(u)v_i:=-\tdiv_s\bigl(a_i(u)\grad_sv_i\bigr) 
\qa v_i\in C^2(\Om), 
$$ 
and, setting 
\hb{a:=(a_1,\ldots,a_n)}, 
$$ 
\cA_s(u)v 
:=-\tdiv_s\bigl(a(u)\grad_sv\bigr) 
:=\bigl(\cA_{1,s}(u)v_1,\ldots,\cA_{n,s}(u)v_n\bigr) 
$$ 
for 
\hb{v=(v_1,\ldots,v_n)\in C^2(\Om,\BR^n)}. Note that $\cA_s(u)$~is a diagonal 
operator whose diagonal elements are coupled by the 
\hbox{$u$-dependence} of their coefficients. 

\smallskip 
If $\gF(\Om;s)$ stands for one of the spaces 
$$ 
W_{\coW p}^k(\Om;s),\ BC^k(\Om;s),\text{ or }B_p^{2-2/p}(\Om;s),
\text{ then }\gF(\Om,\BR^n;s):=\gF(\Om;s)^n. 
$$ 
Given subsets $A$ and~$B$ of some topological spaces, 
\hb{A\isis B} means that $\ol{A}$~is compact and contained in the interior 
of~$B$. 

\smallskip 
We define 
\begin{equation}\label{Q.V} 
V:=\bigl\{\,v\in B_p^{2-2/p}(\Om,\BR^n;s) 
\ ;\ v(\Om)\isis X\,\bigr\}. 
\end{equation}  
\setcounter{lemma}{0} 
\begin{lemma}\label{lem-Q.V} 
If 
\hb{p>m+2}, then $V$~is open in~$B_p^{2-2/p}(\Om,\BR^n;s)$. 
\end{lemma} 
\begin{proof} 
It follows from Lemma~\ref{lem-T.em} that 
\begin{equation}\label{Q.em} 
B_p^{2-2/p}(\Om,\BR^n;s)\hr BC^1(\Om,\BR^n;s)\hr BC(\Om,\BR^n). 
\end{equation}  
Denote the embedding operator which maps the leftmost space into the 
rightmost one by~$\ia$. Let 
\hb{u_0\in V} so that 
\hb{u_0(\Om)\isis X}. 

\smallskip 
Fix\footnote{%
\hb{\dist\bigl(u_0(\Om),\es\bigr):=\iy}.} 
\hb{0<r<\dist\bigl(u_0(\Om),\pl X\bigr)} and set 
\begin{equation}\label{Q.K} 
K:=\bigl\{\,x\in X 
\ ;\ \dist\bigl(x,u_0(\Om)\bigr)<r\,\bigr\}\isis X 
\end{equation}  
and 
$$ 
B(u_0,r):=\bigl\{\,u\in BC(\Om,\BR^n) 
\ ;\ \|u-u_0\|_\iy<r\,\bigr\}. 
$$ 
Then 
\hb{u(\Om)\is K} for 
\hb{u\in B(u_0,r)}. Hence $B(u_0,r)$ is a neighborhood of~$u_0$ 
in $BC(\Om,K)$. Thus 
\hb{\ia^{-1}\bigl(BC(\Om,K)\bigr)} is a neighborhood of~$u_0$ 
in $B_p^{2-2/p}(\Om,\BR^n;s)$ and it is contained 
in~$V$. This proves the claim. 
\end{proof} 
For abbreviation, 
$$ 
E_0:=L_p(\Om,\BR^n;s) 
\qb E_1:=W_{\coW p}^2(\Om,\BR^n;s) 
\qb E:=B_p^{2-2/p}(\Om,\BR^n;s). 
$$ 
As usual, $\cL\EeEn$ is the Banach space of bounded linear operators 
from~$E_1$ into~$E_0$, and $C^\idmi{1}$~means `locally Lipschitz 
continuous'. 
\begin{lemma}\label{lem-Q.A} 
Suppose 
\hb{p>m+2}. Then 
$$ 
\bal 
{\rm(i)}\quad 
&\cA_s(u_0)\text{ has maximal $L_p(\Om,\BR^n;s)$ 
 regularity for }u_0\in V.\cr 
{\rm(ii)}\quad 
&\bigl(u\mt\cA_s(u)\bigr)\in C^\idmi{1}\bigl(V,\cL\EeEn\bigr). 
\eal 
$$ 
\end{lemma} 
\begin{proof} 
We denote by~$c$ constants~% 
\hb{\geq1} which may be different from occurrence  to occurrence and write 
\hb{BC_s^1:=BC^1(\Om,\BR^n;s)} and 
\hb{BC:=BC(\Om,\BR^n)}. 

\smallskip 
Let 
\hb{u_0\in V}. Fix a bounded neighborhood~$V_K$ of~$u_0$ in 
\hb{\ia^{-1}\bigl(BC(\Om,K)\bigr)\is E}. This is possible by the 
preceding lemma.

\smallskip 
(1) 
It is a consequence  of \eqref{Q.ass}(iii) and \eqref{Q.K} that 
\begin{equation}\label{Q.ac} 
1/c\leq a_i(u)\leq c 
\qa i=1,\ldots,n 
\qb u\in V_K. 
\end{equation}  
Moreover, \eqref{Q.ass}(iii) also implies that $a_i$ and its Fr\'echet 
derivative~$\pl a_i$ are locally Lipschitz continuous. Hence 
\begin{equation}\label{Q.L} 
\bal 
{}% do not remove!  
&a_i\text{ and $\pl a_i$ are bounded and uniformly}\cr 
\noalign{\vskip-1\jot} 
&\text{Lipschitz continuous on }\oO\times K 
\eal 
\end{equation}  
(e.g.,~\cite[Proposition~6.4]{Ama90g}). From this and \eqref{Q.em} we infer 
that 
\begin{equation}\label{Q.aL} 
\|a_i(u)-a_i(v)\|_\iy\leq c\,\|u-v\|_E 
\qa 1\leq i\leq n 
\qb u,v\in V_K. 
\end{equation}  

\smallskip 
(2) 
Let $i$ and~$j$ run from~$1$ to~$n$ and $\al$~from $1$ to~$m$. Then, 
using the summation convention writing 
\hb{u=(u^1,\ldots,u^n)}, 
\begin{equation}\label{Q.da} 
\pl_\al\bigl(a_i(u)\bigr)(x) 
=(\pl_\al a_i)\bigl(x,u(x)\bigr) 
+(\pl_{m+j}a_i)\bigl(x,u(x)\bigr)\pl_\al u^j(x) 
\end{equation}  
for 
\hb{x\in\Om}. By~\eqref{Q.em}, $V_K$~is bounded in~$BC_s^1$. 
From this, \eqref{Q.da}, and \eqref{Q.L} it follows  
\begin{equation}\label{Q.U} 
\sup_U\big|\pl_\al\bigl(a_i(u)\bigr)\big|\leq c 
\qa u\in V_K. 
\end{equation}  
Similarly, by employing local coordinates on~$\Ga$, 
\begin{equation}\label{Q.g} 
\sup_S\big|\grad_\Ga\bigl(a_i(u)\bigr)\big|_{T\Ga}\leq c 
\qa u\in V_K, 
\npbd 
\end{equation}  
where  
\hb{\vsdot_{T\Ga}} is the vector bundle norm on the tangent bundle~$T\Ga$ 
of~$\Ga$. 

\smallskip 
Let 
\hb{x\in S} and 
\hb{\al=1} in~\eqref{Q.da}. Then we get from 
\hb{0<\rho^s(x)\leq1}, the boundedness of~$V_K$ in~$BC_s^1$, 
and \eqref{Q.aL} that 
\begin{equation}\label{Q.S} 
\sup_S\big|\rho^s\pl_\nu\bigl(a_i(u)\bigr)\big| 
\leq c(1+\sup_S|\rho^s\pl_\nu u|)\leq c 
\qa u\in V_K. 
\end{equation}  
By collecting \eqref{Q.ac} and \hbox{\eqref{Q.U}--\eqref{Q.S}}, 
%\eqref{Q.g}  
we find 
(cf.~\eqref{T.Ck} and \eqref{T.SU}) 
\begin{equation}\label{Q.bc} 
a(u)\in BC_s^1 
\qb \|a(u)\|_{BC_s^1}\leq c 
\qa u\in V_K. 
\npbd 
\end{equation}  
Now (i)~follows from \eqref{Q.ac} and Corollary~\ref{cor-T.MR}. 

\smallskip 
(3) 
Let 
\hb{u,v\in V}. Then 
$$ 
\bal 
{}% do not remove! 
\pl_\al\bigl(a_i(u)-a_i(v)\bigr) 
&=\pl_\al(a_i)(u)-(\pl_\al a_i)(v)\cr 
&\ph{={}}
 +\bigl((\pl_{m+j}a)(u)-(\pl_{m+j}a)(v)\bigr)\pl_\al u^j\cr 
&\ph{={}}
 +(\pl_{m+j}a)(v)(\pl_\al u^j-\pl_\al v^j). 
\eal 
$$ 
Using \eqref{Q.L} and \eqref{Q.em}, we obtain 
$$ 
\sup_U|(\pl_\al a)(u)-(\pl_\al a)(v)| \leq c\,\|u-v\|_E 
\qa u,v\in V_K. 
$$ 
Similarly, employing also the boundedness of~$V_K$ in~$E$ and \eqref{Q.bc}, 
$$ 
\sup_U\big|\bigl(\pl_{m+j}a)(u)-(\pl_{m+j}a)(v)\bigr)(\pl_\al u^j)\big| 
\leq c\,\|u-v\|_E 
$$ 
and 
$$ 
\sup_U|(\pl_{m+j}a)(v)(\pl_\al u^j-\pl_\al v^j)| 
\leq c\,\|u-v\|_E 
$$ 
for 
\hb{u,v\in V_K}. Consequently, 
$$ 
\sup_U\big|\pl_\al\bigl(a(u)-a(v)\bigr)\big|\leq c\,\|u-v\|_E 
\qa u,v\in V_K. 
$$ 

\smallskip 
By analogous arguments we obtain, as in  step~(1), 
$$ 
\sup_S\big|\grad_\Ga\bigl(a(u)-a(v)\bigr)\big|_{T\Ga}\leq c\,\|u-v\|_E 
$$ 
and 
$$ 
\sup_S\big|\rho^s\pl_\nu\bigl(a(u)-a(v)\bigr)\big|\leq c\,\|u-v\|_E 
$$ 
for 
\hb{u,v\in V_K}. In summary and recalling \eqref{Q.aL}, 
$$ 
\|a(u)-a(v)\|_{BC_s^1}\leq c\,\|u-v\|_E 
\qa u,v\in V_K. 
\npbd 
$$ 
This implies claim~(ii). 
\end{proof} 
We also suppose 
\begin{equation}\label{Q.g1} 
g\in C^1(\oO\times X,\,\BR^{n\times n}) 
\end{equation}  
and define~$g(u)$ analogously to \eqref{Q.ai}. Then, using obvious 
identifications,  
\begin{equation}\label{Q.fg} 
f(u):=g(u)u 
\qa u\in C(\Om,X). 
\end{equation} 
\addtocounter{remark}{1} 
\begin{remark}\label{rem-Q.f} 
This assumption on the production rate in~\eqref{I.cl} is motivated 
by models from population dynamics. It means that the reproduction 
(birth or death) rate is proportional to the size of the actually present 
crowd. Already the diagonal form 
$$ 
f_i(u)=g_i(u)u_i 
\qa 1\leq i\leq n, 
$$ 
comprises the most frequently studied ecological models, namely the 
standard (two-population) models with competing (predator--prey 
or cooperative) species, for example. In those cases the~$g_i$ are affine 
functions of~$u$.\hfill$\qed$ 
\end{remark} 
\addtocounter{lemma}{1} 
\begin{lemma}\label{lem-Q.f} 
Let 
\hb{p>m+2}. Then 
$$ 
\bigl(u\mt f(u)\bigr)\in C^\idmi{1}\bigl(V,L_p(\Om\times\BR^n;s)\bigr). 
$$ 
\end{lemma} 
\begin{proof} 
Let 
\hb{u_0\in V} and fix~$V_K$ as in the preceding proof. Then it is obvious 
from \eqref{Q.em} and \eqref{Q.g1} that 
$$ 
\|g(u)\|_\iy\leq c 
\qa u\in V_K, 
$$ 
and 
$$ 
\|g(u)-g(v)\|_\iy\leq c\,\|u-v\|_E 
\qa u,v\in V_K. 
$$ 
Thus, since 
\hb{E\hr L_p(\Om,\BR^n;s)=:L_{p,s}}, 
$$ 
\|f(u)\|_{L_{p,s}}\leq\|g(u)\|_\iy\,\|u\|_{L_{p,s}}\leq c 
\qa u\in V_K, 
$$ 
and 
$$ 
\bal 
\|f(u)-f(v)\|_{L_{p,s}} 
&\leq\|g(u)-g(v)\|_\iy\,\|u\|_{L_{p,s}} 
 +\|g(v)\|_\iy\,\|u-v\|_{L_{p,s}}\cr 
&\leq c\,\|u-v\|_E  
\eal 
\npbd 
$$ 
for 
\hb{u,v\in V_K}. 
\end{proof} 
Now we can prove the main result of this paper, 
a general well-posedness theorem for strong 
\hbox{$L_p(\Om,\BR^n;s)$~solutions}, by simply referring to known results. 
The reader may 
consult \cite{Ama90g} or~\cite{Ama93a} for definitions and the facts on 
semiflows to which we appeal. 
\addtocounter{theorem}{2} 
\begin{theorem}\label{thm-Q.RD} 
Let \eqref{Q.ass}, \eqref{Q.g1}, and \eqref{Q.fg} be satisfied and assume 
\hb{p>m+2}. Define~$V$ by \eqref{Q.V}. 
Then the initial value problem for the \hbox{$s$-degenerate} 
quasilinear reaction-diffusion system 
\begin{equation}\label{Q.q} 
\bal 
\pl_tu-\tdiv_s\bigl(a(u)\grad_su\bigr)  &=f(u)    
  &&\quad\text{on }    &\Om    &\times\BR_+,\cr
\ga_0u                                  &=u_0  
  &&\quad\text{on }    &\Om    &\times\{0\}, 
\eal 
\end{equation}  
has for each 
\hb{u_0\in V} a~unique maximal solution 
$$ 
u(\cdot,u_0)\in W_{\coW p}^{(2,1)}\bigl(\Om\times[0,t^+(u_0)),X;s\bigr). 
$$ 
The map 
\hb{(t,u_0)\mt u(t,u_0)} is a locally Lipschitz continuous semiflow on~$V$. 
The exit time $t^+(u_0)$ is characterized by the following three (non 
mutually exclusive) alternatives: 
$$ 
\bal 
{\rm(i)}\quad 
&t^+(u_0)=\iy.\cr 
{\rm(ii)}\quad 
&\liminf_{t\ra t^+(u_0)}\dist\bigl(u(t,u_0)(\Om),\pl X\bigr)=0.\cr 
{\rm(iii)}\quad 
&\lim_{t\ra t^+(u_0)}u(t,u_0) 
 \text{ does not exist in }B_p^{2-2/p}(\Om,\BR^n;s).
\eal 
$$ 
\end{theorem}
\begin{proof} 
Due to Lemmas \ref{lem-Q.V}, \ref{lem-Q.A}, and~\ref{lem-Q.f}, this follows 
from Theorem~5.1.1 and Corollary~5.1.2 in J.~Pr\"{u}ss and 
G.~Simonett~\cite{PS16a}. 
\end{proof} 
\addtocounter{remarks}{5} 
\begin{remarks}\label{rem-Q.RD}  
(a) 
It is obvious from the above proofs that the regularity assumptions for $a$ 
and~$g$, concerning the variable 
\hb{x\in\Om}, are stronger than actually needed. We leave it to the 
interested reader to find out the optimal assumptions. 

\smallskip 
(b) 
Suppose that $a_i(u)$~is independent of~$u_j$ for 
\hb{j\neq i}. Then the theorem remains true, with the obvious definitions 
of the weighted spaces, if we replace~$\cA_{i,s}(u)$ by~$\cA_{i,s_i}(u_i)$ 
with 
\hb{1\leq s_i<\iy} for 
\hb{1\leq i\leq n}. 
\begin{proof} 
This follows by an inspection of the proof of Lemma~\ref{lem-Q.A}.
\end{proof} 
(c) 
For simplicity, we have assumed that 
\hb{\Ga=\pO}. It is clear that we can also consider the case where 
$\Ga$~is a proper open and closed subset of~$\pO$ and regular boundary 
conditions are imposed on the remaining part. 

\smallskip 
(d) 
Similar results can be proved for strongly coupled systems, so-called 
cross-diffusion equations.\hfill$\qed$ 
\end{remarks} 
\section{Examples and Remarks}\label{sec-E} 
We close this paper by presenting some easy examples. In addition, 
we include some remarks on open problems and suggestions for further 
research. Throughout this section, 
$$ 
\bt\quad 
1\leq s<\iy 
\quad\text{and}\quad 
p>m+2. 
$$ 
\begin{example}\label{exa-Q.ex} 
\uti{Two-population models} 
Let 
$$ 
a,b\in C^2(\oO,\BR_+) 
\qb a_i,b_i\in C^1(\oO), 
\ i=0,1,2 
\qb \al,\ba,\ga,\da\in\BR. 
$$ 
Consider the \hbox{$s$-degenerate} quasilinear system 
\begin{equation}\label{E.2} 
\bal 
\pl_tu-\tdiv_s\bigl((a+u^\al v^\ba)\grad_su\bigr) 
&=(a_0+a_1u+a_2v)u,\cr 
\pl_tv-\tdiv_s\bigl((b+u^\ga v^\da)\grad_sv\bigr) 
&=(b_0+b_1v+b_2u)v
\eal 
\npbd 
\end{equation}  
on 
\hb{\Om\times\BR_+}. 

\smallskip 
Suppose 
\hb{\ve>0} and
$$ 
(u_0,v_0)\in W_{\coW p}^2(\Om;s) 
\qa u_0,v_0\geq\ve. 
$$ 
Then there exist a maximal 
\hb{t^+=t^+(u_0,v_0)\in(0,\iy]} and a unique solution 
$$ 
(u,v)\in W_{\coW p}^{(2,1)}\bigl(\Om\times[0,t^+),\BR^2;s\bigr) 
\npbd 
$$ 
of~\eqref{E.2} satisfying 
\hb{u(t)(x)>0} and 
\hb{v(t)(x)>0} for 
\hb{x\in\Om} and 
\hb{0\leq t<t^+}. 
\end{example} 
\begin{proof} 
Theorem~\ref{thm-Q.RD} with 
\hb{X=(0,\iy)^2}. 
\end{proof} 
Observe that the right side of~\eqref{E.2} encompasses standard 
predator--prey as well as cooperation models, depending on the signs of the 
coefficient functions. 

\smallskip 
In the following examples we restrict ourselves to scalar equations. 
\setcounter{examples}{1} 
\begin{examples}\label{exa-E.ex} 
(a)    
\uti{Porous media equations} 
Let 
\hb{\al\in\BR\ssm\{0\}} and assume that 
\hb{g\in C^1(\oO\times\BR)}. Then 
$$ 
\pl_tu-\tdiv_s(u^\al\grad_su)=g(u)u 
\quad\text{ in }\Om\times\BR_+ 
$$ 
has for each 
\hb{u_0\in W_{\coW p}^2(\Om;s)} with 
\hb{u_0\geq\ve>0} a~unique maximal solution 
$$ 
u\in W_{\coW p}^{(2,1)}\bigl(\Om\times(0,t^+);s\bigr) 
\npbd 
$$ 
satisfying 
\hb{u(t)(x)>0} for 
\hb{x\in\Om} and 
\hb{0\leq t<t^+}.
\begin{proof} 
Theorem~\ref{thm-Q.RD} with 
\hb{X:=(0,\iy)}. 
\end{proof} 
(b) 
\uti{Diffusive logistic equations} 
Assume 
\hb{\al,\lda\in\BR} with 
\hb{\al>0}. Let 
\hb{a\geq0}. Set 
$$ 
\Da_s:=\tdiv_s\grad_s. 
$$ 
If 
$$ 
u_0\in W_{\coW p}^2(\Om;s) 
\qa u_0\geq\ve>0, 
$$ 
then there exist a maximal 
\hb{t^+\in(0,\iy]} and a unique solution 
$$ 
u\in W_{\coW p}^{(2,1)}\bigl(\Om\times[0,t^+);s\bigr) 
$$ 
of 
\begin{equation}\label{E.l} 
\pl_tu-\al\Da_su=(\lda-au)u 
\quad\text{on }\Om\times\BR_+, 
\npbd 
\end{equation}  
satisfying 
\hb{u(t)(x)>0} for 
\hb{x\in\Om} and 
\hb{0\leq t<t^+}. 
\begin{proof} 
This is essentially a subcase of Example~\ref{exa-E.ex}. 
\end{proof} 
\end{examples} 
The most natural question which now arises is: 
$$ 
\bt\quad 
\text{How can we prove global existence?}
$$ 
An attempt to tackle this challenge, which is already hard in the case of 
standard boundary value problems, is even more demanding in the 
present setting. To point out where some of the difficulties originate 
from, we 
review in the following remarks some of the well-known techniques, 
which have successfully been applied to parabolic boundary value problems, 
and indicate why they do not straightforwardly apply to 
\hbox{$s$-degenerate} problems. 
\setcounter{remark}{2} 
\begin{remark}\label{rem-E.M}  
\uti{Maximum principle techniques} 
First we look at the diffusive logistic equation~\eqref{E.l} and contrast 
it with the simple classical counterpart 
\begin{equation}\label{E.cl} 
\bal 
\pl_tu-\al\Da u &=(\lda-u)u &&\quad\text{on }    &\oO    &\times\BR_+,\cr
u               &=0         &&\quad\text{on }    &\Ga    &\times\BR_+.  
\eal 
\end{equation}  
Suppose 
\hb{\lda>0}. Then \eqref{E.cl} has for each sufficiently smooth 
initial value~$u_0$ satisfying 
\hb{0\leq u_0\leq\lda} a~unique global solution obeying the same bounds 
as~$u_0$. This is a consequence of the maximum principle, since 
$(0,\lda)$ is a pair of sub- and supersolutions. For the validity of this 
argument it is crucial that we deal with a boundary value problem. 

\smallskip 
In the \hbox{$s$-degenerate} case there is no boundary. Hence the preceding 
argument does not work, since there is no appropriate maximum 
principle.\hfill$\qed$ 
\end{remark} 
\begin{remark}\label{rem-E.S} 
\uti{Methods based on spectral properties} 
A~further important technique, which is useful in the case of boundary 
value problems, rests on spectral properties of the linearization of 
associated stationary elliptic equations. The most prominent case 
is supposedly the `principle of linearized stability' (and its 
generalizations to non-isolated equilibria, see \cite[Chapter~5]{PS16a}). 
In addition, the better part of the qualitative studies on semi- and 
quasilinear parabolic boundary value problems, as well for a single 
equation as for systems, is based on spectral properties, in particular 
on the existence and the nature of eigenvalues. 

\smallskip 
In the case  of boundary value problems, the associated linear elliptic 
operators have compact resolvents, due to the compactness of~$\oO$. In our 
situation, $\Om$,~more precisely, the Riemannian manifold 
\hb{\Om_s=(\Om,g_s)} introduced in Remark~\ref{rem-S.g}, is not compact. If 
\hb{s=1}, then it is a manifold with cylindrical ends in the sense of 
R.~Melrose~\cite{Mel93a} and others\footnote{I am grateful to Victor 
Nistor for pointing this out to me.}. For manifolds of this type---and 
more general ones---much is known about the \hbox{($L_2$)~spectrum} of the 
Laplace--Beltrami operator. In particular, the essential spectrum 
is not empty. 

\smallskip 
Nevertheless, we are in a simpler situation. In fact, the normal collar~$S$ 
can be represented as a half-cylinder over the compact manifold~$(\Ga,h)$ 
and the remaining `interior part'~$U$ is flat 
(cf.~\cite[Section~5]{Ama20a}). Thus there is hope to get sufficiently 
detailed information on the \hbox{$L_p$~spectrum} of linear divergence 
form operators. 

\smallskip 
In particular, suppose 
\hb{u_0\in V\cap W_{\coW p}^2(\Om,\BR^n;s)} is a stationary point 
of~\eqref{Q.q}. If it can be shown that the spectrum of~$\cA(u_0)$ 
is contained in an interval $[\al,\iy)$ with 
\hb{\al>0}, then, using the decay properties of the analytic semigroup 
generated by~$-\cA(u_0)$, it can be shown that $u_0$~is an asymptotically 
stable critical point of~\eqref{Q.q}.\hfill$\qed$ 
\end{remark} 
\begin{remark}\label{rem-E.A} 
\uti{The technique of a~priori estimates} 
The general version of \cite[Theorem~5.1.1]{PS16a} exploits the 
regularization properties of analytic semigroups. Suppose that 
alternative~(ii) of Theorem~\ref{thm-Q.RD} does not occur. Also assume that 
there can be established a uniform a~priori bound in a Besov space 
\hb{B_{p,s}^{\sa-2/p}:=B_p^{\sa-2/p}(\Om,\BR^n;s)} with 
\hb{2/p<\sa<2}. Then we have global existence, \emph{provided} 
the embedding 
\hb{B_{p,s}^{2-2/p}\hr B_{p,s}^{\sa-2/p}} is compact (see 
\cite[Theorem~5.7.1]{PS16a}). Hence this technique is also not applicable 
to our equations. 

\smallskip 
However, we still have the possibility to use the 
interpolation-extrapolation techniques developed in \cite{Ama93a} and 
\cite{Ama95a} to switch to weak formulations. Then it might be possible 
to prove global existence  by more classical techniques using  
a~priori estimates with respect to suitable integral norms.\hfill$\qed$ 
\end{remark} 
{\small 
 \def\cprime{$'$} \def\polhk#1{\setbox0=\hbox{#1}{\ooalign{\hidewidth
  \lower1.5ex\hbox{`}\hidewidth\crcr\unhbox0}}}
\providecommand{\bysame}{\leavevmode\hbox to3em{\hrulefill}\thinspace}

\noindent 
Herbert Amann\\
Math.\ Institut, Universit\"at Z\"urich, Winterthurerstr.~190,\\     
CH 8057 Z\"urich, Switzerland, herbert.amann@math.uzh.ch 
}
\end{document}